\documentclass[graybox]{svmult}

\usepackage{amsfonts}
\usepackage{amsmath}
\usepackage{mathptmx}       % selects Times Roman as basic font
\usepackage{helvet}         % selects Helvetica as sans-serif font
\usepackage{courier}        % selects Courier as typewriter font
\usepackage{type1cm}        % activate if the above 3 fonts are
\usepackage{makeidx}         % allows index generation
\usepackage{graphicx}        % standard LaTeX graphics tool
\usepackage{multicol}        % used for the two-column index
\usepackage[bottom]{footmisc}% places footnotes at page bottom

\makeindex             % used for the subject index
\def\RR{\mathbb R}

\def\NN{\mathbb N}

\def\E{\mathcal E}

\newcommand{\rd}{\mathbb{R}^{d}}

\begin{document}

\title*{On the Hausdorff dimension of graphs of prevalent continuous functions on compact sets}
\titlerunning{On the Hausdorff dimension of graphs}
\author{Fr\'ed\'eric Bayart and Yanick Heurteaux}
\institute{Fr\'ed\'eric Bayart \at Clermont Universit\'e, Universit\'e Blaise Pascal, Laboratoire de Math\'ematiques, BP 10448, F-63000 CLERMONT-FERRAND -
CNRS, UMR 6620, Laboratoire de Math\'ematiques, F-63177 AUBIERE, \email{Frederic.Bayart@math.univ-bpclermont.fr}
\and Yanick Heurteaux \at Clermont Universit\'e, Universit\'e Blaise Pascal, Laboratoire de Math\'ematiques, BP 10448, F-63000 CLERMONT-FERRAND -
CNRS, UMR 6620, Laboratoire de Math\'ematiques, F-63177 AUBIERE, \email{Yanick.Heurteaux@math.univ-bpclermont.fr}}
\maketitle

\abstract{Let $K$ be a compact set in $\rd$ with positive Hausdorff dimension. Using a Fractional Brownian Motion, we prove that in a prevalent set of continuous functions on $K$,   the Hausdorff dimension of the graph is equal to $\dim_{\mathcal H}(K)+1$. This is the largest possible value. This result generalizes a previous work due to J.M. Fraser and J.T. Hyde (\cite{FH}) which was exposed in the conference {\it Fractal and Related Fields~2}. The case of $\alpha$-H\"olderian functions is also discussed.}
\section{Introduction}
Let $d\ge 1$ and let $K$ be a compact subset in $\RR^d$. Denote by $\mathcal C(K)$ the set of continuous functions on $K$ with real values.
This is a Banach space when equipped with the supremum norm, $\|f\|_\infty=\sup_{x\in K} |f(x)|$. The graph of a function $f\in\mathcal C(K)$ is the set
$$\Gamma_f^K=\left\{ (x,f(x))\ ;\ x\in K\right\}\subset\RR^{d+1}.$$
It is often difficult to obtain the exact value of the Hausdorff dimension of the graph $\Gamma_f^K$ of a precised continuous function $f$. For example, a famous conjecture says that the Hausdorff dimension of the graph of the Weierstrass function
$$f(x)=\sum_{k=0}^{+\infty}2^{-k\alpha}\cos(2^kx),$$
where $0<\alpha<1$, satisfies
$$\dim_{\mathcal H}\left(\Gamma_f^{[0,2\pi]}\right)=2-\alpha.$$
This is the natural  expected value, but, to our knowledge, this conjecture is not yet solved. 

If we add some randomness, the problem becomes much easier and Hunt proved in \cite{Hunt} that the Hausdorff dimension of the graph of the random Weierstrass function 
$$f(x)=\sum_{k=0}^{+\infty}2^{-k\alpha}\cos(2^kx+\theta_k)$$
where $(\theta_k)_{k\ge 0}$ is a sequence of independent uniform random variables is almost surely equal to the expected value $2-\alpha$.

\medskip
In the same spirit we can hope to have a generic answer to the following question:
\begin{eqnarray*}
\textrm{``What is the Hausdorff dimension of the graph of a continuous function?''}
\end{eqnarray*}

Curiously, the answer to this question depends on the type of genericity we consider. If genericity is relative to the Baire category theorem, Mauldin and Williams proved at the end of the 80's the following result:
\begin{theorem}\label{THMBAIRE}{\em (\cite{MW})}
For quasi-all functions $f\in \mathcal C([0,1])$, we have
$$\dim_{\mathcal H} \left(\Gamma_f^{[0,1]}\right)=1.$$
\end{theorem}
This statement on the Hausdorff dimension of the graph is very  surprising because it seems to say that a generic continuous function is quite regular. Indeed it is convenient to think that there is a deep correlation between strong irregularity properties of a function  and large values of the Hausdorff dimension of its graph.

\medskip
This curious result seems to indicate that genericity in the sense of the Baire category theorem is not ``the good notion of genericity'' for this question. In fact, when genericity is related to the notion of prevalence (see Section \ref{SECPRE} for a precise definition), Fraser and Hyde recently obtained the following result.
\begin{theorem}\label{THMFH}{\em (\cite{FH})}
Let $d\in\NN^*$. The set
$$\left\{ f\in\mathcal C([0,1]^d)\ ;\ \dim_{\mathcal H}\left(\Gamma_f^{[0,1]^d}\right)=d+1\right\}$$
is a prevalent subset of $\mathcal C([0,1]^d)$.
\end{theorem}
This result says that the Hausdorff dimension of the graph of a generic continuous function
 is as large as possible and is much more in accordance with the idea that a generic continuous function is strongly irrregular. 

\medskip
The main tool in the proof of Theorem \ref{THMFH} is the construction of a fat Cantor set in the interval $[0,1]$ and a stochastic process on $[0,1]$ whose graph  has almost surely  Hausdorff dimension 2.
 This construction is difficult to generalize in a compact set $K\not=[0,1]$. 
Nevertheless, there are in the litterature stochastic processes whose almost sure Hausdorff dimension of their graph is well-known. The most famous example is the Fractional Brownian Motion. 
Using such a process, we are able to prove the following generalisation of Theorem \ref{THMFH}.
\begin{theorem}\label{THMMAIN}
 Let $d\ge 1$ and let $K\subset \rd$ be a compact set such that $\dim_{\mathcal H}(K)>0$. The set
$$\left\{ f\in\mathcal C(K)\ ;\ \dim_{\mathcal H}\left(\Gamma_f^K\right)=\dim_{\mathcal H}(K)+1\right\}$$
is a prevalent subset of $\mathcal C(K)$.
\end{theorem}
In this paper, we have decided to focus to the notion of Hausdorff dimension of graphs. Nevertheless, we can mention that there are also many papers that deal with the generic value of the dimension of graphs when the notion of dimension is for example the lower box dimension (see \cite{FF,GJ,HL,Shaw}) or the packing dimension (see \cite{HP,McClure}).

\medskip
The paper is devoted to the proof of Theorem \ref{THMMAIN} and is organised as follows. In Section \ref{SECPRE} we recall the basic facts on prevalence. In particular we explain how to use a stochastic process in order to prove prevalence in functional vector spaces. In Section \ref{SECW}, we prove an auxiliary result on Fractional Brownian Motion which will be the key of the main theorem. We finish the proof of Theorem \ref{THMMAIN} in Section \ref{SECMAIN}. Finally, in a last section, we deal with the case of $\alpha$-H\"olderian functions.

\section{Prevalence}\label{SECPRE}
Prevalence is a notion of genericity who generalizes in infinite dimensional vector spaces the notion of ``almost everywhere with respect to Lebesgue measure''. This notion has been 
introduced by J. Christensen in \cite{Chr72} and has been
widely studied since then. In fractal and multifractal analysis, some properties which are true on a dense $G_\delta$-set are also prevalent 
(see for instance \cite{FJK}, \cite{FJ06} or \cite{BH11b}), whereas some are not (see for instance 
\cite{FJK} or \cite{Ol10}). 
\begin{definition}
Let $E$ be a complete metric vector space. A Borel set $A\subset E$ is called 
\emph{Haar-null} if there exists a compactly supported 
probability measure $\mu$ such that, for any $x\in E$, $\mu(x+A)=0$. If this property 
holds, the measure $\mu$
is said to be \emph{transverse} to $A$.\\
A subset of $E$ is called \emph{Haar-null} if it is contained in a Haar-null Borel set. 
The complement of a Haar-null set
is called a \emph{prevalent} set.
\end{definition}
The following results enumerate important properties of prevalence and show that this 
notion supplies a natural generalization
of "almost every" in infinite-dimensional spaces:
\begin{itemize}
\item If $A$ is Haar-null, then $x+A$ is Haar-null for every $x\in E$.
\item If $\dim(E)<+\infty$, $A$ is Haar-null if and only if it is negligible
  with respect to the  Lebesgue measure.
\item Prevalent sets are dense.
\item The intersection of a countable collection of prevalent sets is prevalent.
\item If $\dim(E)=+\infty$, compacts subsets of $E$ are Haar-null.
\end{itemize}
In the context of a functional vector space $E$, a usual way to prove that a set $A\subset E$ is prevalent is to use a stochastic process. More precisely, suppose that $W$ is a stochastic process defined on a probability space $(\Omega,\mathcal F, \mathbb P)$ with values in $E$ and satisfies. 
$$\forall f\in E,\quad f+W\in A\quad\text{almost surely}.$$
Replacing $f$ by $-f$, we get that the law $\mu$ of the stochastic process $W$ is such that
$$\forall f\in E,\quad \mu(f+A)=1.$$
In general, the measure $\mu$ is not compactly supported. Nevertheless, if we suppose that the vector space $E$ is also a Polish space (that is if we add the hypothesis that $E$ is separable), then we can find a compact set $Q\subset E$ such that $\mu(Q)>0$. It follows that the compactly supported probability measure $\nu=(\mu(Q))^{-1}\mu_{|Q}$ is transverse to $E\setminus A$.

\section{On the graph of a perturbed Fractional Brownian Motion}\label{SECW}
In this section, we prove an auxilliary result which will be the key of the proof of Theorem \ref{THMMAIN}. For the definition and the main properties of the Fractional Brownian Motion, 
we refer to \cite[Chapter 16]{Falc}.
\begin{theorem}\label{THMFBM}
Let $K$ be a compact set in $\rd$ such that 
$\dim_{\mathcal H}(K)>0$ and $\alpha\in (0,1)$. Define the stochastic process in $\rd$ 
\begin{eqnarray}\label{EQNPROCESS}
W(x)=W^1(x_1)+\cdots+ W^d(x_d)
\end{eqnarray}
where $W^1,\cdots,W^d$ are independent Fractional Brownian Motions starting from 0 with Hurst parameter equal to $\alpha$. Then, for any function $f\in\mathcal C(K)$
$$\dim_{\mathcal H}\left(\Gamma_{f+W}^K\right)\ge\min\left(\frac{\dim_{\mathcal H}(K)}{\alpha}\,,\, \dim_{\mathcal H}(K)+1-\alpha\right)\quad\textrm{almost surely}.$$
 \end{theorem}

Let us remark that the conclusion of Theorem \ref{THMFBM} is sharp. More precisely, suppose that $f=0$ and let $\varepsilon>0$. It is well known that the Fractional Brownian Motion is almost-surely uniformly $(\alpha-\varepsilon)$-H\"olderian. It follows that the stochastic process $W$ is also 
uniformly $(\alpha-\varepsilon)$-H\"olderian on $K$. It is then straightforward that the graph $\Gamma_W^K$ satisfies
$$\dim_{\mathcal H}\left(\Gamma_W^K\right)\le \dim_{\mathcal H}(K)+1-(\alpha-\varepsilon)\quad\mbox{a.s.}.$$
On the other hand, the application
$$\Phi\ :\ x\in K\longmapsto (x,W(x))\in{\mathbb R}^{d+1}$$
is almost-surely $(\alpha-\varepsilon)$-H\"olderian. It follows that
$$\dim_{\mathcal H}\left(\Gamma_W^K\right)\le \frac{\dim_{\mathcal H}(K)}{\alpha-\varepsilon}\quad\mbox{a.s.}.$$
\smallskip

The proof of Theorem \ref{THMFBM} is based on the following lemma.

\begin{lemma}\label{LEMESP}
Let $s>0$, $\alpha\in (0,1)$ and $W$ be the process defined as in {\em (\ref{EQNPROCESS})}. There exists a constant $C:=C(s)>0$ such that for any $\lambda\in\RR$, for any $x,y\in\rd$, 
$$\mathbb E\left[\frac1{(\Vert x-y\Vert^2+(\lambda+W(x)-W(y))^2)^{s/2}}\right]\le
C\left\{
\begin{array}{ll}
 \displaystyle \Vert x-y\Vert^{1-s-\alpha}&\textrm{ provided }s>1\\
\displaystyle \Vert x-y\Vert^{-\alpha s}&\textrm{ provided }s<1.
\end{array}
\right.$$
\end{lemma}
\begin{proof}
 Observe that $W(x)-W(y)$ is a centered gaussian variable with variance  $$\sigma^2=h_1^{2\alpha}+\cdots+h_d^{2\alpha}$$
where $h=(h_1,\cdots,h_d)=x-y$. H\"older's inequality yields 
$$\Vert h\Vert^{2\alpha}\le\sigma^2\le d^{1-\alpha}\Vert h\Vert^{2\alpha}.$$
Now,
$$\mathbb E\left[\frac1{(\Vert x-y\Vert^2+(\lambda+W(x)-W(y))^2)^{s/2}}\right]=\int\frac{e^{-u^2/(2\sigma^2)}}{(\Vert h\Vert^2+(\lambda+u)^2)^{s/2}}\frac{du}{\sigma\sqrt{2\pi}}.$$
Suppose that $s>1$. We get
\begin{eqnarray*}
\mathbb E\left[\frac1{(\Vert x-y\Vert^2+(\lambda+W(x)-W(y))^2)^{s/2}}\right]&\le&\int\frac{du}{(\Vert h\Vert^2+(\lambda+u)^2)^{s/2}\sigma\sqrt{2\pi}}\\
&=&\int\frac{\Vert h\Vert\,dv}{(\Vert h\Vert^2+(\Vert h\Vert v)^2)^{s/2}\sigma\sqrt{2\pi}}\\
&\le&\Vert h\Vert^{1-s-\alpha}\frac1{\sqrt{2\pi}}\int\frac{dv}{(1+v^2)^{s/2}}\\
&:=& C\,\|x-y\|^{1-s-\alpha}.
\end{eqnarray*}
In the case where $0<s<1$, we write
\begin{eqnarray*}
\mathbb E\left[\frac1{(\Vert x-y\Vert^2+(\lambda+W(x)-W(y))^2)^{s/2}}\right]&\le&\int\frac{e^{-v^2/2}}{(\lambda+\sigma v)^{s}}\frac{dv}{\sqrt{2\pi}}\\
&\le&\Vert h\Vert^{-\alpha s}\int\frac{e^{-v^2/2}}{(\gamma+v)^s}\frac{dv}{\sqrt{2\pi}}
\end{eqnarray*}
where $\gamma=\lambda\sigma^{-1}$.
On the other hand, 
\begin{eqnarray*}
\int\frac{e^{-v^2/2}\,dv}{(\gamma+v)^s}=\int\frac{e^{-(v-\gamma)^2/2}\,dv}{v^s}&\le&\int_{-1}^1\frac{dv}{v^s}+\int_{\RR\setminus[-1,1]}\frac{e^{-(v-\gamma)^2/2}\,dv}{v^s}\\
&\le& \int_{-1}^1\frac{dv}{v^s}+\int_\RR e^{-x^2/2}\,dx
\end{eqnarray*}
which is a constant $C$ independent of $\gamma$ and $\alpha$.
\end{proof}
\smallskip

We are now able to finish the proof of Theorem \ref{THMFBM}. We use the potential theoretic approach (for more details on the potential theoretic approach of the calculus of the Hausdorff dimension,
 we can refer to \cite[Chapter 4]{Falc}). Suppose first that $\dim_{\mathcal H}(K)>\alpha$ and let $\delta$ be a real number such that 
$$\alpha<\delta<\dim_{\mathcal H}(K).$$ 
%Frostman's Lemma ensures the existence of a probability measure $m$ on $K$ satisfying
% $$\forall x\in K\quad\forall r>0,\quad m(B(x,r)\le Cr^{\delta'}$$
% where $B(x,r)$ is the ball with center $x$ and radius $r$ in $\rd$. It is then well known that the $\delta$-energy of $m$
% $$I_\delta(m)=\int\!\!\int_{K\times K}\frac{dm(x)\,dm(y)}{\Vert x-y\Vert^\delta}$$
% is finite.
There exists a probability measure $m$ on $K$ whose $\delta$-energy $I_\delta(m)$, defined by
 $$I_\delta(m)=\int\!\!\int_{K\times K}\frac{dm(x)\,dm(y)}{\Vert x-y\Vert^\delta}$$
is finite. Conversely, to prove that the Hausdorff dimension of the graph  $\Gamma_{f+W}^K$ is at least $\dim_{\mathcal H}(K)+1-\alpha$, it suffices to find, for any $s<\dim_{\mathcal H}(K)+1-\alpha$,
a measure $\mu$ on $\Gamma_{f+W}^K$ with finite $s$-energy. 

Let $(\Omega,\mathcal F, \mathbb P)$ be the probability space where are defined the Fractional Brownian Motions $W^1,\cdots,W^d$. For any $\omega\in\Omega$, define  $m_\omega$ as the image of the measure $m$ on the graph $\Gamma_{f+W_\omega}^K$ via the natural projection $$x\in K\longmapsto (x,f(x)+W_\omega(x)).$$
Set $s=\delta+1-\alpha$ which is greater than 1.  The $s$-energy of $m_\omega$ is equal to
\begin{eqnarray*}
I_s(m_\omega)&=&\int\!\!\int_{\Gamma_{f+W_\omega}^K\times\Gamma_{f+W_\omega}^K}\frac{dm_\omega(X)\,dm_\omega(Y)}{\Vert X-Y\Vert^s}\\
&=&\int\!\!\int_{K\times K}\frac{dm(x)\,dm(y)}{\Big(\Vert x-y\Vert^2+\big(f(x)+W_\omega(x)-(f(y)+W_\omega(y))\big)^2\Big)^{s/2}}.
\end{eqnarray*}
Fubini's theorem and Lemma \ref{LEMESP} ensure that
\begin{eqnarray*}
 \mathbb E\left[I_s(m_\omega)\right]&=&\int\!\!\int_{K\times K}\mathbb E\left[ \frac{1}{\Big(\Vert x-y\Vert^2+\big((f(x)-f(y))+(W(x)-W(y))\big)^2\Big)^{s/2}}\right]dm(x)\,dm(y)\\
&\le&C\int\!\!\int_{K\times K}\Vert x-y\Vert^{1-s-\alpha}dm(x)\,dm(y)\\
&=&CI_\delta(m)\\
&<&+\infty.
\end{eqnarray*}
We deduce that for $\mathbb P$-almost all $\omega\in \Omega$, the energy $I_s(m_\omega)$ is finite. Since $s$ can be chosen arbitrary closed to $\dim_{\mathcal H}(K)+1-\alpha$, we get 
$$\dim_{\mathcal H}\left(\Gamma_{f+W_\omega}^K\right)\ge \dim_{\mathcal H}(K)+1-\alpha\quad 
\mbox{almost surely}.$$

In the case where $\dim_{\mathcal H}(K)\le\alpha$, we proceed exactly in the same way, except that we take any $\delta<\dim_{\mathcal H}(K)$ and we set $s=\frac\delta\alpha$ which is smaller than 1. We
then get 
$$\dim_{\mathcal H}\left(\Gamma_{f+W}^K\right)\ge\frac{\dim_{\mathcal H}(K)}{\alpha}\quad\mbox{ almost surely}.$$

\section{Proof of Theorem \ref{THMMAIN}}\label{SECMAIN}
We can now prove Theorem \ref{THMMAIN}. Let $K$ be a compact set in $\rd$ satisfying $\dim_{\mathcal H}(K)>0$. Remark first that for any function $f\in\mathcal C(K)$,
 the graph $\Gamma_f^K$ is included in $K\times\RR$. It follows that
$$\dim_{\mathcal H}\left(\Gamma_f^K\right)\le \dim_{\mathcal H}(K\times \RR)= \dim_{\mathcal H}(K)+1.$$
Define
$$G=\left\{ f\in\mathcal C(K);\ \dim_{\mathcal H}\left(\Gamma_f^K\right)=\dim_{\mathcal H}(K)+1\right\}.$$
Theorem \ref{THMFBM} says that for any $\alpha$ such that $0<\alpha<\min(1,\dim_{\mathcal H}(K))$, the set $G_\alpha$ of all continuous functions $f\in\mathcal C(K)$ satisfying $\dim_{\mathcal H}\left(\Gamma_f^K\right)\ge\dim_{\mathcal H}(K)+1-\alpha$ is prevalent in $\mathcal C(K)$. Finally, we can write
$$G=\bigcap_{n\ge 0}G_{\alpha_n}$$
where $(\alpha_n)_{n\ge 0}$ is a sequence decreasing to 0 and we obtain that $G$ is prevalent in $\mathcal C(K)$.
\begin{remark}
It is an easy consequence of Ascoli's theorem that the law of the process $W$ is compactly supported in $\mathcal C(K)$ (remember that $W$ is almost surely $(\alpha-\varepsilon)$-H\"olderian). Then,
 we don't need to use that $\mathcal C(K)$ is a Polish space to obtain Theorem \ref{THMMAIN}.
\end{remark}
\begin{remark}
 Let $K=[0,1]$ and $f\in \mathcal C([0,1])$. Theorem \ref{THMMAIN} implies that the set $G\bigcap(f+G)$ is prevalent. We can then write
$$f=f_1-f_2\quad\text{with}\quad\dim_{\mathcal H}\left(\Gamma_{f_1}^{[0,1]}\right)=2\quad\text{and}\quad\dim_{\mathcal H}\left(\Gamma_{f_2}^{[0,1]}\right)=2$$
where $f_1$ and $f_2$ are continuous functions.

On the other hand, it was recalled in Theorem \ref{THMBAIRE} that the set
$$\tilde G=\left\{f\in\mathcal C([0,1])\ ;\ \dim_{\mathcal H}\left(\Gamma_f^{[0,1]}\right)=1\right\}$$
contains a dense $G_\delta$-set of $\mathcal C([0,1])$. It follows that any continuous function $f\in\mathcal C([0,1])$ can be written
$$f=f_1-f_2\quad\text{with}\quad\dim_{\mathcal H}\left(\Gamma_{f_1}^{[0,1]}\right)=1\quad\text{and}\quad\dim_{\mathcal H}\left(\Gamma_{f_2}^{[0,1]}\right)=1$$
where  $f_1$ and $f_2$ are continuous functions.

We can then ask the following question: given a real number $\beta\in(1,2)$ can we write an arbitrary  continuous function $f\in\mathcal C([0,1])$ in the following way
$$f=f_1-f_2\quad\text{with}\quad\dim_{\mathcal H}\left(\Gamma_{f_1}^{[0,1]}\right)=\beta\quad\text{and}\quad\dim_{\mathcal H}\left(\Gamma_{f_2}^{[0,1]}\right)=\beta$$
where  $f_1$ and $f_2$ are continuous functions?

We do not know the answer to this question.
\end{remark}
\section{The case of $\alpha$-H\"olderian functions}\label{SECHOLDER}
Let $0<\alpha<1$ and let $\mathcal C^\alpha(K)$ be the set of $\alpha$-H\"olderian functions in $K$ endowed with the standard norm
$$\Vert f\Vert_\alpha=\sup_{x\in K}|f(x)|+\sup_{(x,y)\in K^2}\frac{|f(x)-f(y)|}{\| x-y\|^\alpha}.$$
It is well known that the Hausdorff dimension of the graph $\Gamma_f^K$ of a function $f\in\mathcal C^\alpha(K)$ satisfies
\begin{eqnarray}\label{EQNCALPHA}
\dim_{\mathcal H}\left(\Gamma_f^K\right)\le\min\left(\frac{\dim_{\mathcal H}(K)}{\alpha}\,,\,\dim_{\mathcal H}(K)+1-\alpha\right)
\end{eqnarray}
(see for example the remark following the statement of Theorem \ref{THMFBM}). It is then natural to ask if inequality (\ref{EQNCALPHA}) is an equality in a prevalent set of $\mathcal C^\alpha(K)$. This is indeed the case as said in the following result.
\vfill\eject
\begin{theorem}\label{THMALPHA}
 Let $d\ge 1$, $0<\alpha<1$ and $K\subset \RR^d$ be a compact set with strictly positive Hausdorff dimension. The set
$$\left\{f\in\mathcal C^\alpha(k)\ ;\ \dim_{\mathcal H}\left(\Gamma_f^K\right)=\min\left(\frac{\dim_{\mathcal H}(K)}{\alpha}\,,\,\dim_{\mathcal H}(K)+1-\alpha\right)\right\}$$
is a prevalent subset of $\mathcal C^\alpha(K)$.
\end{theorem}
This result generalizes in an arbitrary compact subset of $\RR^d$ a previous work of Clausel and Nicolay (see \cite[Theorem 2]{CN}).
\begin{proof}
Let $\alpha<\alpha'<1$ and let $W$ be the stochastic process defined in Theorem \ref{THMFBM} with Hurst parameter $\alpha'$ instead of $\alpha$. The stochastic process $W_{|K}$ takes values in $\mathcal C^\alpha(K)$. Moreover, if $\alpha<\alpha''<\alpha'$, the injection
$$f\in\mathcal C^{\alpha''}(K)\longmapsto f\in\mathcal C^{\alpha}(K)$$
is compact. It follows that the law  of the stochastic process $W_{|K}$ is compactly supported in $\mathcal C^{\alpha}(K)$ ($W$ is $\alpha''$-H\"olderian). Then, Theorem \ref{THMFBM} ensures that the set 
$$\left\{f\in\mathcal C^\alpha(K)\ ;\ \dim_{\mathcal H}\left(\Gamma_f^K\right)\ge\min\left(\frac{\dim_{\mathcal H}(K)}{\alpha'}\,,\,\dim_{\mathcal H}(K)+1-\alpha'\right)\right\}$$
is prevalent in $\mathcal C^\alpha(K)$. Using a sequence $(\alpha_n)_{n\ge 0}$ decreasing to $\alpha$, we get the conclusion of Theorem \ref{THMALPHA}.
\end{proof}

\end{document}